\documentclass[a4paper,11pt]{amsart}
\usepackage[all,arc]{xy}
\usepackage[all]{xy}
\usepackage{latexsym}
\usepackage{amssymb}
\usepackage{amsmath}
\usepackage{amsthm}

\newtheorem{definition}{Definition}[section]
\newtheorem{lemma}[definition]{Lemma}
\newtheorem{prop}[definition]{Proposition}
\newtheorem{theorem}[definition]{Theorem}
\newtheorem{cor}[definition]{Corollary}

\newtheorem{remark}[definition]{Remark}
\theoremstyle{definition}
\newtheorem{exam}[definition]{Example}
\newtheorem{fact}[definition]{Fact}

\newcommand*{\ann}{\mathop{{\rm ann}}\nolimits}
\newcommand*{\Ass}{\mathop{{\rm Ass}}\nolimits}

\newcommand{\DPR}{\mathop{{\rm DPR}}\nolimits}

\newcommand*{\Inv}{\mathop{{\rm Inv}}\nolimits}
\newcommand*{\Ker}{\mathop{{\rm Ker}}\nolimits}

\newcommand*{\lcm}{\mathop{{\rm lcm}}\nolimits}

\newcommand*{\rad}{\mathop{{\rm rad}}\nolimits}

\newcommand{\Zg}{\mathop{{\rm Zg}}\nolimits}

\newcommand*{\Div}{\mathop{\rm{Div}}\nolimits}

\newcommand*{\seq}{\subseteq}

\newcommand*{\bsm}{\left(\begin{smallmatrix}}
\newcommand*{\esm}{\end{smallmatrix}\right)}
\newcommand*{\bp}{\begin{pmatrix}}
\newcommand*{\ep}{\end{pmatrix}}

\newcommand*{\wg}{\wedge}

\newcommand{\midd}{ | }

\newcommand*{\Ra}{\Rightarrow}

\newcommand*{\N}{\mathbb{N}}

\renewcommand*{\phi}{\varphi}

\date{}

\begin{document}

\footskip=30pt

\title[]{Decidability of the theory of modules over
Pr\"{u}fer domains with infinite residue fields}

\author[]{Lorna Gregory}

\address[L.~Gregory]{University of Camerino, School of Science and Technologies,
Division of Mathematics, Via Madonna delle Carceri 9, 62032 Camerino, Italy}

\email{lorna.gregory@gmail.com}

\author[]{Sonia L'Innocente}

\address[S.~L'Innocente]{University of Camerino, School of Science and Technologies,
Division of Mathematics, Via Madonna delle Carceri 9, 62032 Camerino, Italy}

\email{sonialinnocente@unicam.it}

\author[]{Gena Puninski}

\address[G.~Puninski]{Belarusian State University, Faculty of Mechanics and Mathematics, av. Nezalezhnosti 4,
Minsk 220030, Belarus}


\author[]{Carlo Toffalori}

\address[C.~Toffalori]{University of Camerino, School of Science and Technologies,
Division of Mathematics, Via Madonna delle Carceri 9, 62032 Camerino, Italy}

\email{carlo.toffalori@unicam.it}

\thanks{The second and fourth authors were supported by Italian PRIN 2012 and GNSAGA-INdAM}

\thanks{Gena Puninski died on April 29, 2017, while this paper was
being completed. The other three authors would like to dedicate
this article as a tribute to his memory.}


\begin{abstract}
We provide algebraic conditions ensuring the decidability of
the theory of modules over effectively given Pr\"ufer (in particular B\'ezout)
domains with infinite residue fields in terms of a suitable generalization of the prime radical
relation. For B\'{e}zout domains these conditions are also necessary.
\end{abstract}

\maketitle

\pagestyle{plain}

\vspace{3mm}

\section{Introduction}

We deal here with decidability of first order theories of modules over Pr\"ufer (in particular B\'ezout) domains $R$ with infinite residue fields. We assume $R$ {\sl effectively given} (so countable), in order to  ensure that the decision problem for $R$-modules makes sense.

The model theory of modules over B\'ezout domains, with some hints at Pr\"ufer domains, is studied in \cite{PT15}. The decidability of the theory of modules over the ring of algebraic integers is proved in \cite{LPT} (see also \cite{LP}), and a similar result is obtained in \cite{PT14} over B\'ezout domains obtained from principal ideal domains by the so called D+M-construction \cite{F-S}.

On the other hand Gregory \cite{Gre15}, extending \cite{PPT}, proved that the theory of modules over a(n effectively given) valuation domain $V$ is decidable if and only if there is an algorithm which decides the prime radical relation, namely, for every $a, b \in V$, answers whether $a \in \rad(bV)$ (equivalently, whether the prime ideals of $V$ containing $b$ also include $a$).

This paper develops a similar analysis in a closely related setting, that is, over Pr\"ufer domains. In fact a domain is Pr\"ufer if and only if all its localizations at maximal ideals are valuation domains. B\'ezout domains are a notable subclass of Pr\"ufer domains. In both cases we focus on the domains all of whose residue fields with respect to maximal ideals are infinite. The reason and the benefit of this choice are illustrated in $\S$ 3 below. Notice that Pr\"ufer (indeed B\'ezout) domains with infinite residue fields include the ring of algebraic integers and the ring of complex valued entire functions - even if the latter is uncountable and so cannot be effectively given (but see the analysis of its Ziegler spectrum in \cite{LPPT}). Other noteworthy examples will be proposed in $\S$ 6.

Our main result, resembling \cite{Gre15}, states that, if $R$ is such a B\'ezout domain, then the theory of $R$-modules is decidable if and only if there is an algorithm which answers a sort of double prime radical relation, in detail, given $a, b, c, d \in R$, decides whether, for all  prime ideals $\mathfrak{p},\mathfrak{q}$ with $\mathfrak{p}+\mathfrak{q}\neq R$, $b \in \mathfrak{p}$ implies $a \in \mathfrak{p}$ or $d \in\mathfrak{q}$ implies $c \in \mathfrak{q}$. This will be proved in $\S$ 6. Generalizations to Pr\"ufer domains will be presented in the final part of the paper, in $\S$ 7. The  preceding sections $\S \S$ 2-5 describe the framework of (effectively given) Pr\"ufer domains and prepare the main theorems.

We refer to all the already mentioned papers and books, as well as to the key references on model theory of modules, \cite{Preb1}, \cite{Preb2} and \cite{Zi}. We also assume some familiarity with Pr\"ufer domains, as treated, for instance, in \cite{F-S} and \cite{Gi}.  \lq \lq Domain" means commutative domain with unity, and \lq \lq module" abbreviates right unital module, unless otherwise stated.

We thank the referee for her/his valuable comments and suggestions.

\section{Pr\"ufer domains}

First let us summarize some basic facts on the model theory of modules over Pr\"ufer, and in particular B\'ezout, domains.

Recall that a domain is Pr\"ufer if all its localizations at maximal ideals, and consequently at non-zero prime ideals, are valuation domains.

A domain $R$ is said to be \emph{B\'ezout}, if every 2-generated ideal (and consequently every finitely generated ideal) is principal. Thus $R$ is B\'ezout if and only if the so called B\'ezout identity holds: for every $0 \neq a, b\in R$ there are $c, u, v, g, h \in R$ such that $au+ bv= c$ and $cg=a$, $ch= b$ hold. Then $c$ is called a \emph{greatest common divisor} of $a$ and $b$, written $\gcd(a,b)$, and is unique up to a multiplicative unit.

B\'ezout domains are GCD domains, \cite[p. 17]{F-S}, and hence, \cite[4.5]{F-S}, the intersection of two principal ideals is also principal. For every $0 \neq a, b \in R$, if $aR \cap bR = dR$, then $d$ is said to be a \emph{least common multiple} of $a$ and $b$, written $\lcm(a,b)$. This least common multiple is again unique up to a multiplicative unit. Thus for $0\neq a, b\in R$, under a suitable choice of units, we obtain the equality $ab= \gcd(a,b)\cdot \lcm(a,b)$.

B\'ezout domains are Pr\"ufer.

Let $L_R$ denote the first order language of modules over any commutative ring $R$.  If $a \in R$ then $a \midd x$ denotes the \emph{divisibility formula} of $L_R$, which defines in a module $M$ the submodule $Ma$. Similarly the
\emph{annihilator formula} $xb=0$  ($b \in R$) defines in $M$ the submodule $\{m\in M\midd mb=0\}$. Let $T_R$ be the $L_R$-theory of $R$-modules.

Positive primitive formulae (pp-formulae from now on) play a crucial role in the model theory of modules. Over a Pr\"ufer domain they admit the following {\sl normal} form.

\begin{fact}\label{normal}  (\cite[Fact 2.2]{PT15}) Every pp-formula $\phi(x)$ over a Pr\"ufer domain $R$ is equivalent to a finite sum of formulae $\exists y \, (ya = x \land yb = 0)$, and also to a finite conjunction of formulae $c \midd xd$ (with $a, b, c, d \in R$).
\end{fact}

Note that the formula $\exists y \, (ya = x \land yb = 0)$ is the elementary dual, see \cite{Hr}, of the formula $b\midd ax$.

Over B\'ezout domains one can say more.

\begin{fact}\label{2-1} (\cite[Lemma 2.3]{PT15})
Let $R$ be a B\'ezout domain. Then every pp-formula $\phi(x)$ of $L_R$ is equivalent in $T_R$ to a finite sum of formulae $a\midd x\wg xb=0$, $a, b\in R$, and to a finite conjunction of formulae $c\midd x+ xd=0$, $c, d\in R$.
\end{fact}

In the above, and throughout this paper, $c\midd x+xd=0$ stands for $(c\midd x)+(xd=0)$, the sum of $c\midd x$ and $xd=0$.

The representation in Fact \ref{2-1} is obtained by using $\gcd$ - a tool we cannot rely on over arbitrary Pr\"ufer domains. However the following result by Tuganbaev provides some help also in this enlarged setting.

\begin{fact}\label{tuganbaev} If $R$ is a Pr\"{u}fer domain, then for all $a, b\in R$ there exist $\alpha, r, s \in R$ such that $a \alpha = b r$ and $b (\alpha - 1)= a s$.
\end{fact}

In fact \cite[Lemma 1.3]{Tu} (specialised to the case $R=M$) shows that if $R$ is a right distributive ring (i.e., when viewed as a right module over itself, it has distributive lattice of submodules) then for all $a, b \in R$ there exists $\alpha \in R$ such that $a \alpha \in b R$ and $b (\alpha -1) \in aR$. On the other hand Pr\"{u}fer domains are exactly the commutative distributive domains.

The {\sl Ziegler spectrum} of $R$, $\Zg_R$, is a topological space whose points are (isomorphism classes of) indecomposable pure injective $R$-modules, and whose topology is given by basic open sets of the form $(\phi / \psi)$ where $\phi$ and $\psi$ ranges over pp-formulae of $L_R$ in one free variable. Recall that an open set $(\phi / \psi)$ consists of the $R$-modules $N$ in $\Zg_R$ such that $\phi(N)$ strictly includes its intersection with $\psi(N)$. Moreover the endomorphism ring of a module $N$ in $\Zg_R$ is local (see \cite[Theorem 4.27]{Preb1}, for instance).

The lattice of pp-$1$-formulae of a Pr\"{u}fer domain is distributive \cite[3.1]{EkHe}. Thus, \cite[3.3]{PunKGserial} implies the following fact which we will use repeatedly.

\begin{fact}\label{pp-uni}
If $R$ is a Pr\"{u}fer domain and $N$ is an indecomposable pure injective $R$-module then $N$ is pp-uniserial i.e. its lattice of pp-definable subgroups is totally ordered.
\end{fact}


Over any ring $R$ and for every choice of pp-formulae $\phi_i (x)$ and $\psi_j(x)$ ($i \leq n$, $j \leq m)$ in one free variable $x$ we have the following equality of open subsets of $\Zg_R$.
$$(\dagger) \quad \quad(\sum _{i \leq n} \phi_i \, / \, \bigwedge_{j \leq m} \psi_j) = \bigcup_{i \leq n, \, j \leq m} (\phi_i \, / \, \psi_j).$$

Combined with Fact \ref{2-1} this gives us the following for B\'ezout domains.

\begin{lemma}\label{uniserial} (\cite[Cor. 4.1]{PT15})
Over a B\'ezout domain a basic open set $(\phi / \psi)$ of $\Zg_R$ is the finite
union of open sets $(a \midd x \wg xb = 0 \, / \, c \midd x + xd = 0)$.
\end{lemma}

Using once again $\gcd$ and $\lcm$ we may further assume that $c= ga$ and $b= dh$ for some $g, h\in R$. Clearly this open set is empty if and only if either some element among $a$, $d$, $g$, $h$ is 0 or $g$ and $h$ are coprime, that is, $\gcd(g,h)= 1$.





Combined with Fact \ref{normal}, $(\dagger)$ gives the following for Pr\"{u}fer domains.

\begin{lemma}\label{2-3}
Over a Pr\"{u}fer domain a basic open set $(\phi / \psi)$ of $\Zg_R$ is the finite union of open sets $( \exists y \, (ya = x \land yb = 0) \, / \, c \midd xd)$.
\end{lemma}

The role of these open sets $(a \midd x \wg xb = 0 \, / \, c \midd x + xd = 0)$ is crucial even over arbitrary Pr\"ufer domains. In fact, thanks to \cite{Tu} and Fact \ref{tuganbaev}, the following can be shown.

\begin{lemma}\label{extension}
Let $R$ be a Pr\"ufer domain and $a, b, c, d, \in R$. Let $\alpha, s, r \in R$ satisfy $a \alpha = br$ and $b (\alpha -1 ) = as$, and similarly let $\delta, t, u \in R$ satisfy $d \delta = ct$ and $c (\delta -1) = du$. Then

\[ \left(\frac{ \exists y (x = ya \wedge yb = 0) }{c \midd xd}\right)
= \]
\[ =
\left(\frac{ a \midd x \wedge xs=0 }{u \midd x + xd = 0}
\right)
\cap
\left[
\left(\frac{ x \alpha = 0}{x=0} \right)
\cup
\left(\frac{ x=x }{\alpha \midd x } \right)
\right]
\cap
\left[
\left(\frac{ x \delta = 0}{x=0} \right)
\cup
\left(\frac{ x=x }{\delta \midd x } \right)
\right]
.\]

\end{lemma}

\begin{proof}
Note that $N\in\left(\frac{x\alpha=0}{x=0}\right)\cup\left(\frac{x=x}{\alpha\midd x}\right)$ if and only if $\alpha$ acts non-invertibly on $N$.

If $\alpha$ acts invertibly on $N$, then $\exists y (x = ya \wedge yb = 0)$ is equivalent to $x = 0$ in $N$, and consequently implies $c \midd xd$. Namely, if $m, n \in N$ satisfy $m = na$ and $nb = 0$, then $m \alpha = n a \alpha = n b r = 0$, and consequently $m = 0$. Thus, if $N$ is in the left hand set then $\alpha$ acts non-invertibly on $N$.

Similarly, if $\delta$ acts invertibly on $N$, then $c \midd xd$ is equivalent to $x = x$ in $N$, and consequently is implied by $\exists y (x = ya \wedge yb  = 0)$. Namely, every $m \in N$ satisfies $m d \delta = m c t$, whence $c \midd m d \delta$ and (as $\delta$ acts invertibly) $c \midd md$. Thus, if $N$ is in the left hand set then $\delta$ acts non-invertibly on $N$.

So we have shown that if $N$ is in the left hand set then $N$ is in the second and third conjunct of the right hand side. Moreover, if $N$ is in either the right hand set or the left hand set then $\alpha$ and $\delta$ act non-invertibly on $N$. Thus, if $N$ is in either the right hand set or the left hand set then, since the ring of endomorphisms of $N$ is local, $\alpha -1$ and $\delta - 1$ act invertibly on $N$.

\textbf{Claim 1}: If $\beta\in R$ acts invertibly on $N$ and $c\beta=du$ then $c\midd xd$ is equivalent to $u\midd x+xd=0$ in $N$.

Suppose $m,n\in N$ and $nc=md$. Since $\beta$ acts invertibly on $N$ there exists $n'\in N$ such that $n'\beta=n$. Thus $n'du=n'\beta c=md$. So $(n'u-m)d=0$ and hence $m$ satisfies $u\midd x+xd=0$.

Conversely, if $m\in N$ satisfies $u\midd x+xd=0$ then there exists $n,l\in N$ such that $m=nu+l$ and $ld=0$. So $md=nud=n\beta c$. Thus $c\midd md$.

\textbf{Claim 2}: If $\gamma\in R$ acts invertibly on $N$ and $as=b\gamma$ then $\exists y (x = ya \wedge yb = 0)$ is equivalent to $a\midd x\wedge xs=0$ in $N$.

Suppose $m,n\in N$, $m=na$ and $nb=0$. Then $ms=nas=nb\gamma=0$. So $m$ satisfies $a\midd x\wedge xs=0$.

Conversely, suppose that $m,n\in N$, $m=na$ and $ms=0$. Then $nb\gamma=nas=0$. Since $\gamma$ acts invertibly on $N$, $nb=0$. Thus $m$ satisfies $a\midd x\wedge xs=0$.

Since we have shown that if $N$ is in either the right hand set or the left hand set then $\alpha -1$ and $\delta - 1$ act invertibly on $N$, then applying claim $1$ with $\beta=\delta-1$ and claim 2 with $\gamma=\alpha-1$ finishes the proof.
\end{proof}

Let $N$ be an $R$-module. Define
\[\Ass N:=\{r\in R \, \mid \, \text{ there exists } m\in N\backslash\{0\} \text{ with } mr=0\}\]
and
\[\Div N:= \{r \in R \, \mid \, r  \not | \, m \; \text{ for some } m\in N\} . \]

\begin{lemma}\label{prime} Let $R$ be a Pr\"ufer domain and $N$ an indecomposable pure injective $R$-module.
Then $\Ass N$ and $\Div N$ and their union $\Ass N\cup\Div N$ are (proper) prime ideals of $R$.
\end{lemma}
\begin{proof}
First we deal with $\Ass N$. It is easily seen that it is closed under multiplication by arbitrary elements of $R$ and excludes the unity 1 of $R$. In order to show closure under addition, we use Fact \ref{pp-uni}
($N$ is pp-uniserial). Hence take $r, r' \in \Ass N$ with corresponding $m, m' \in M \backslash \{ 0 \}$ such that $mr =m'r'=0$. By pp-uniseriality in
$N$ either $\Ker \, r \subseteq \Ker \, r'$ or $\Ker \, r \supseteq \Ker \, r'$. If the latter holds then $m' r=0$ and hence $m(r + r') = 0$. Thus $r+r' \in \Ass N$. The other case is symmetric. Finally let $r, r' \in R$ with $rr' \in \Ass N$. If $m\in N\backslash\{0\}$ and $mrr'=0$ then either $mr=0$ or $mr\neq 0$ and $(mr)r'=0$. Thus $rr'\in\Ass N$ implies $r\in \Ass N$ or $r'\in \Ass N$.

The proof for $\Div N$ is similar. Clearly $\Div N$ is closed under multiplication by elements of $R$ and does not contain 1. Furthermore, if
$r, r' \in \Div N$ the the same is true of $r+r'$. In fact by pp-uniseriality $Nr \supseteq Nr'$ or $Nr \subseteq Nr'$. Assume the latter. Then $N(r + r') \subseteq Nr'$ and any element $m \in N$, $m \notin Nr'$ is also out of $N(r + r')$. Finally let $r, r' \in R$ with $rr' \in \Div N$. Take $m \in N \backslash Nrr'$. If $m \in Nr$, whence $m = nr$ for some $n \in N$, then $n \notin Nr'$.

The set $\Ass N \cup \Div N$ is a prime ideal because when working over a commutative ring, the set of elements that, for some given indecomposable pure injective module $N$, do no act as automorphisms on $N$ is a prime ideal. Clearly $\Ass N \cup \Div N$ exclude 1 and hence is a proper ideal of $R$.
\end{proof}

We now recall the correspondence, over a valuation domain $R$, between ordered pairs of proper ideals of $R$ and indecomposable pp-types in one variable over $R$. The indecomposable pp-type associated to an ordered pair $(I,J)$ of ideals is just the unique complete pp-type $p = p(I, J)$ such that, for all $r \in R$,
\begin{itemize}
\item $xr = 0 \in p$ if and only if $r \in I$ and
\item $r \midd x \in p$ if and only if $r \notin J$,
\end{itemize}
see \cite[3.4]{EkHe}. Note that the consistency conditions required there become vacuous when $R$ is a valuation domain.
Through indecomposable pp-types, pairs of ideals correspond to indecomposable pure injective $R$-modules. The equivalence relation linking two pairs $(I,J)$ and $(K,L)$ if and only if the corresponding indecomposable pure injective $R$-modules realising $p(I,J)$ and $p(K,L)$ are isomorphic is also described in \cite[3.4]{EkHe}.

There is a version of this correspondence for B\'ezout domains described in \cite[Thm. 4.5]{PT15} (see also \cite[$\S$ 4]{LPPT}) but we will not use it in this paper.

\section{Effectively given Pr\"ufer domains}

The decision problem of the theory of modules over a ring $R$ makes sense only when $R$ is {\sl effectively given} (see \cite{PPT} and \cite[Ch. 17]{Preb1}). Let us focus on Pr\"ufer domains and say that such a domain $R$ is  \emph{effectively given} if it is countable and its elements can be listed as $a_0=0, a_1= 1, a_2, \dots$ (possibly with repetitions) so that suitable algorithms effectively execute the following, when $m, n$ range over natural numbers.
\begin{enumerate}
\item Deciding whether $a_m= a_n$ or not.
\item Producing $a_m+ a_n$ and $a_m\cdot a_n$, or rather indices of these elements in the list.
\item Establishing whether $a_m$ divides $a_n$.
\end{enumerate}
The countability assumption on $R$ ensures the countability of the first order language $L_R$. Furthermore if $R$ is written as a list, then each instance in (1)-(3) corresponds to a sentence of $L_R$
of which to check membership to $T_R$
(for instance $a_n = a_m$ holds true if and only if  $\forall x (x a_n = x a_m)
\in T_R$), and hence has to be answered effectively. It is well known that, when $R$ is effectively given, the standard list of axioms of the theory of $R$-modules is recursive, whence $T_R$ is recursively enumerated.

As a consequence of (1)--(3) other familiar procedures can be carried out effectively in a Pr\"ufer domain $R$, such as determining units, calculating additive inverses and (for invertible elements) multiplicative inverses. The same applies to $\gcd$ and $\lcm$, when $R$ is B\'ezout. Over an effectively given Pr\"ufer domain, and with respect to Fact \ref{tuganbaev}, given $a$ and $b$, the corresponding $\alpha$, $r$ and $s$ can again be found by a similar searching procedure. In the worst case, this can be done by a brute force strategy, enumerating all the triples of elements of
$R$ and looking among them for a right one, satisfying the equalities in Fact \ref{tuganbaev}.

Coming back to Fact \ref{2-1} and to the pp-formula $\phi$ in it, the equivalent finite sum of conditions $a\midd x\wg xb=0$, $a, b\in R$, and the equivalent finite conjunction of conditions $c\midd x+ xd=0$, $c, d\in R$, can also be effectively found, and the same is true, in the larger Pr\"ufer setting, of the formulae in Fact \ref{normal}. Once again, this can be done by a brute force procedure, enumerating all the formulae of the given forms implied by $\phi$ in $T_R$, and implying $\phi$ in $T_R$, and looking for the equivalent ones - their existence being guaranteed by the related facts. We will often tacitly use similar arguments in the remainder of this paper.

Recall that a pp-formula $\phi(x)$ in one free variable defines, in every $R$-module $M$, a subgroup $\phi(M)$ called a pp-subgroup of $M$. If $\phi$ and $\psi$ are two such formulae then the corresponding {\sl elementary invariant} $\Inv (M, \phi, \psi)$ is the size of the quotient group $\phi(M) / (\phi(M) \cap \psi(M))$, if finite, and $\infty$ otherwise. These elementary invariants depend only on the elementary equivalence class of $M$, and indeed by the Baur-Monk theorem characterize it -- whence their name. If the residue fields of $R$ are infinite, then it is easily seen that  each elementary invariant is 1 or $\infty$.

It follows  from general theory, basically from the Baur-Monk theorem again (see \cite[Sect. 17]{Preb1} or \cite[Sect. 5]{PPT}), that to prove decidability it suffices to  check the inclusions of the above described basic open sets
$$
(\star) \quad \quad (\phi/\psi)\seq \bigcup_{i=1}^n (\phi_i/\psi_i).
$$

By Lemma \ref{2-3}, over a Pr\"{u}fer domain we may assume that the pp-formulae $\phi$ and $\phi_i$ are of the form $\exists y \, (ya = x \land yb = 0)$ and that the pp-formulae $\psi$ and $\psi_i$ are of the form $c\midd xd$. This is because, as seen in $\S$ 2, every open set $(\sigma / \tau)$ can be (effectively, using $\dagger$) decomposed as a finite union of open sets given by pairs of this kind and hence we may replace both the left hand side and the right hand side of $(\star)$ by a finite union of open sets of the appropriate form. We may further assume that the union on the left consists of a single open set of this kind because the finite union of open sets on the left hand side is contained in the union on the right hand side if and only if each single open set on the left hand side is contained in the union on the right hand side.

Replacing each $\left(\phi_i/\psi_i\right)$ by its representation given by Lemma \ref{extension}, we may assume that the right hand side of $(\star)$ is a finite intersection of finite unions of sets of the form
$ \left( \frac{ a \midd x \, \wedge \, xs=0 }{u \midd x \, + \, xd = 0}
\right)$.
Note that the set $\left( \frac{ x \alpha = 0}{x=0} \right)$ is equal to $\left( \frac{1\midd x \, \wedge \, x \alpha = 0}{0\midd x \, + \, x1=0} \right)$ and the set $\left(\frac{ x=x }{\alpha \midd x } \right)$ is equal to $\left( \frac{1\midd x \, \wedge \, x0=0}{\alpha\midd x \, + \, x1=0} \right)$.

Since $(\phi / \psi)$ is contained in a finite intersection of basic open sets if and only if it is included in each of them, we may assume the right hand side of $(\star)$ is a finite union of sets of the form $ \left( \frac{ a \midd x \, \wedge \, xs=0 }{u \midd x \, + \, xd = 0}
\right)$.

In $\S$ 4 we will further simplify the left hand side of $(\star)$.

Before concluding this section, let us examine how this property of being effectively given is preserved by the Kaplansky-Jaffard-Ohm construction of a B\'ezout domain with a given group of divisibility \cite[Theorem 5.3 p. 113]{F-S}.

Let $\Gamma$ be a lattice ordered abelian group written additively. We say that $\Gamma$ is {\sl effectively given} if its elements can be listed (as for $R$ before) so that suitable algorithms execute the following:
\begin{enumerate}
\item Deciding equality $=$.
\item Calculating the group operation $+$.
\item Calculating the lattice operations $\land$ and $\lor$.
\end{enumerate}
As a consequence the order relation of $\Gamma$ can also be decided.

\begin{prop}\label{group}
Let $\Gamma$ be an effectively given lattice ordered group, $R$ be its associated B\'ezout domain with respect to some effectively given field $K$. Then $R$ can be effectively given.
\end{prop}

\begin{proof}
We follow the Kaplansky-Jaffard-Ohm construction as explained in \cite[Theorem 3.5]{F-S}.

We start building the group ring $K[\Gamma]$ - a domain. Its non-zero elements can be represented as finite formal sums $\xi = \sum_{i \leq t} k_i X^{\gamma_i}$ where $t$ is a non-negative integer, $X$ is an indeterminate, the $k_i$ are non-zero elements of $K$ and the $\gamma_i$ are (finitely many) pairwise different elements of $\Gamma$. The representation is unique up to the order of the $\gamma_i$. The ring operations are the trivial ones. For instance, the product of two non-zero elements  $\xi = \sum_{i \leq t} k_i X^{\gamma_i}$ and $\epsilon = \sum_{j \leq s} h_j X^{\delta_j}$ of $K[\Gamma]$ is $\sum_l (\sum_{\gamma_i + \delta_j = \epsilon_l} k_i h_j) X^{\epsilon_l}$, where the $\epsilon_l$'s range among the elements of $\Gamma$ that can be obtained as sums of some $\gamma_i$ and some $\delta_j$ and the $l$'s index them. Thus the non-zero elements of $K [\Gamma]$ can be recursively listed on the basis of the corresponding lists of $K$ and $\Gamma$. It suffices to enumerate the finite subsets of $\Gamma$ and then the functions from these sets to $K \backslash \{ 0 \}$. Moreover equality can be effectively decided, and the ring operations can be effectively calculated. Actually the content of a non-zero element $\xi$, that is, the lattice meet of its $\gamma_i$, written $c(\xi)$, can be also computed. Incidentally, $K[\Gamma]$ itself can be effectively given, namely divisibility can be decided, too.

Next we form the field of fractions $Q$ of $K[\Gamma]$. Clearly it is effectively given - just apply the usual rules determining equality between quotients and calculating their operations (including division). In this case divisibility is trivially checked. Furthermore the content $c$, as extended from $K(\Gamma)$ to $Q$, that is, by putting, for every $\xi$, $\xi' \in K[\Gamma]$ with $\xi' \neq 0$, $c(\xi \xi'^{-1}) = c(\xi) - c(\xi')$, can be effectively calculated, too.

Now $R$ is introduced as the subring of $Q$ consisting of the elements $\alpha$ for which $c(\alpha) \geq 0_\Gamma$ (the zero element of $\Gamma$). As the content can be effectively computed in $Q$, a list of the elements of $R$ can be extracted from that of $Q$. Equality can be decided and ring operations can be calculated, again because $R$ is a subring of $Q$. To check divisibility between two non-zero elements $\alpha$ and $\alpha'$ of $R$, just calculate their quotient in $Q$ and, looking at its content, check whether it belongs to $R$ or not.
\end{proof}

\section{Basic open sets}

We prepare here the main theorem, that will be stated and proved in $\S$ 6. In particular we simplify the structure of  pp-formulae in $(\star)$. Our arguments will mainly rely on Tuganbaev's result in Fact \ref{tuganbaev} and pp-uniseriality of indecomposable pure injective modules over a Pr\"ufer domain (Fact
\ref{pp-uni}).

Lemma \ref{extension} (and Fact \ref{2-1} over B\'ezout domains) have already produced pp-formulae of a somewhat elementary form. In particular we have seen that we can restrict to basic opens sets  $(a \midd x \wg xb = 0 \, / \, c \midd x + xd = 0)$ on the right hand side of $(\star)$. On the other hand, the left side of $(\star)$ contains finite intersections of these sets, rather than a single one of them. We want to improve this point, and reach a simpler representation of the involved open sets.

The next lemma contributes to the latter objective.

\begin{lemma}\label{decomp}
Let $R$ be a Pr\"ufer domain, $\phi', \phi'', \psi', \psi''$ be pp-formulae of $L_R$ in one free variable. Then
$$ ( \phi' \wedge \phi''  \, / \, \psi' + \psi'') =
(\phi' \, / \, \psi') \cap (\phi' \, / \, \psi'') \cap (\phi'' \, / \, \psi') \cap(\phi'' \, / \, \psi'').
$$
\end{lemma}

\begin{proof}
The inclusion of the left side into the right one is clear. On the other hand, for every $N \in \Zg_R$, by the pp-uniseriality of $N$, $\phi'(N) \cap \phi''(N)$ equals either $\phi'(N)$ or $\phi''(N)$, and similarly $\psi'(N) + \psi''(N)$ coincides with either $\psi'(N)$ or $\psi''(N)$, which proves the inverse inclusion.
\end{proof}

As a consequence:

\begin{cor}\label{decomp1}
Let $R$ be a Pr\"ufer domain. For every $a, b, c, d \in R$,
$$
(a \midd x \wedge xb = 0 \, / \, c \midd x + xd = 0) = (a \midd x / c \midd x) \cap (a \midd x / xd = 0) \cap (xb = 0 / c \midd x) \cap (xb = 0 / xd = 0).
$$
\end{cor}

Thus, by proceeding as in $\S$ 3 we can assume that, in the basic open sets $(\phi_i / \psi_i)$ ($1 \leq i \leq n$) of the right side of $(\star)$, $\phi_i$ is either $a_i \midd x$ or $x b_i = 0$ and $\psi_i$ is either $c_i \midd x$ or $x d_i = 0$.




Now let us deal with the left side and with finite intersections. The following lemma applies to this setting.

\begin{lemma}\label{reduction}
Let $\mathcal{W}, \mathcal{U}_i$ ($1\leq i\leq n$) be open sets of $\Zg_R$, $\phi, \phi',\phi'',\psi, \psi', \psi''$ be pp-formulae (in one free variable). Then
\begin{enumerate}
\item $(\phi' \wedge \phi'' / \psi) \cap \mathcal{W} \subseteq \bigcup_{i=1}^n\mathcal{U}_i \quad$  if and only if $\quad (\phi' / \psi) \cap \mathcal{W} \subseteq (\phi' / \phi'') \cup\bigcup_{i=1}^n\mathcal{U}_i \quad$
and $\quad (\phi'' / \psi) \cap \mathcal{W} \subseteq  (\phi'' / \phi') \cup\bigcup_{i=1}^n\mathcal{U}_i$,
\item $( \phi / \psi' + \psi'') \cap \mathcal{W} \subseteq \bigcup_{i=1}^n\mathcal{U}_i \quad$ if and only if $\quad ( \phi / \psi') \cap \mathcal{W}\subseteq ( \psi'' / \psi') \cup \bigcup_{i=1}^n\mathcal{U}_i \quad$ and $\quad ( \phi / \psi'') \cap \mathcal{W} \subseteq  (\psi' / \psi'') \cup \bigcup_{i=1}^n\mathcal{U}_i$.
\end{enumerate}
\end{lemma}

\begin{proof}
By Lemma \ref{decomp} $(\phi' \wedge \phi'' \, / \, \psi) = (\phi' \, / \, \psi) \cap (\phi'' \, / \, \psi)$ and $(\phi \, / \, \psi' + \psi'') = ( \phi \, / \, \psi') \cap (\phi \, / \, \psi'')$. That said, let us deal with (2), as (1) can be handled by similar arguments.

$( \Rightarrow)$ Suppose that $N \in (\phi / \psi')$. Since $N$ is pp-uniserial, either $\psi''(N) \subseteq \psi'(N)$ or $\psi'(N) \subset \psi''(N)$. In the former case $\psi'(N) + \psi''(N) = \psi'(N)$, whence $N \in (\phi / \psi' + \psi'')$ and consequently $N \in \bigcup_{i=1}^n\mathcal{U}_i$. In the latter case $N \in (\psi'' / \psi')$. Hence $N$ is always in the left side union.

The second condition follows symmetrically.

$(\Leftarrow)$ Suppose now $N \in (\phi / \psi') \cap (\phi / \psi'')$. By pp-uniseriality again, either $\psi'(N) \subseteq \psi''(N)$ or $\psi''(N) \subseteq \psi'(N)$. Thus either $N \notin (\psi''/\psi')$ or $N \notin (\psi' / \psi'')$. In either case $N \in \bigcup_{i=1}^n\mathcal{U}_i$.
\end{proof}

Thanks to these reductions, combined with Lemma \ref{extension}, it is enough for our purposes to effectively check, given basic open sets $\mathcal{U}_i = (\phi_i / \psi_i)$ ($1\leq i\leq n$), whether
\[\left(\frac{ \phi }{\psi}
\right)
\cap
\left[
\left(\frac{ x \alpha = 0}{x=0} \right)
\cup
\left(\frac{ x=x }{\alpha \midd x } \right)
\right]
\cap
\left[
\left(\frac{ x \delta = 0}{x=0} \right)
\cup
\left(\frac{ x=x }{\delta \midd x } \right)\right]\subseteq \bigcup_{i=1}^n\mathcal{U}_i\]
where $\phi$ and $\psi$ are either of the form $a\midd x$ or $xb=0$.

Let us examine the various open sets arising in this way as $(\phi / \psi)$ (but also as $(\phi_i / \psi_i)$). It is here that Fact \ref{tuganbaev} is useful.

\begin{lemma}\label{divis}
Let $R$ be a Pr\"{u}fer domain. Let $a, c, \alpha, r, s \in R$, $a \alpha = cr $ and $c(\alpha - 1)= as$. Then
\[\left(a|x / c|x \right)=\left(x=x/ s | x\right)\cap\left(x=x / xa=0\right)\cap\left(x=x/\alpha|x\right).\]
\end{lemma}
\begin{proof}
Suppose that $N\in\left(a|x / c|x \right)$. Then there exists $m,n\in N\backslash\{0\}$ such that $m=na$ and $c$ does not divide $m$. In particular $N\in \left(x=x / xa=0 \right)$. Moreover, if $n=n' s$, then $m=ab =n'as =n'(\alpha - 1) c$, which contradicts the assumption that $c$ does not divide $m$. Thus $N\in \left(x=x / s|x\right)$. Similarly, if $n=n' \alpha$ then $m=na \alpha=ncr $, which again contradicts the assumption that $c$ does not divide $m$. Thus $ N\in \left(x=x/\alpha|x\right)$. Consequently \[N\in \left(x=x / s|x\right)\cap\left(x=x/xa=0\right)\cap\left(x=x/\alpha|x\right).\]

Conversely suppose that $N\in \left(x=x/ s | x \right)\cap\left(x=x/xa = 0\right)\cap\left(x=x / \alpha|x\right)$. Let $\mathfrak{m}$ be a maximal ideal such that $N$ is an $R_{\mathfrak{m}}$-module. Since $N\in \left(x=x / \alpha | x \right)$, $\alpha\in\mathfrak{m}$. Thus $\alpha - 1 \notin \mathfrak{m}$.

Since $N$ is pp-uniserial either $N s \subseteq \Ker \, a$ or $\Ker \, a \subseteq Ns$.

If $N s \subseteq \Ker \, a$ then $m s a=0$ for all $m\in N$. Thus $mc(\alpha -1 )=0$ for all $m\in N$. Since $\alpha - 1$ acts invertibly on $N$, $mc=0$ for all $m\in N$. On the other hand, since $N \in \left(x=x/xa=0\right)$, there is some $m \in N$ for which $ma \neq 0$. Thus $N\in\left(a|x / c|x \right)$.

Suppose $\Ker \, a \subseteq N s$. If $m,m'\in N$ and $ma=m'c$ then $ma(\alpha -1)=m'c (\alpha-1) = m' a s$. So $(m(\alpha - 1)-m' s) a =0$, that is, $m(\alpha - 1) - m' s$ is in $\Ker \, a$ and consequently both $m(\alpha - 1) - m' s$ and $m(\alpha - 1)$ itself are in $Ns$. Thus even $m(\alpha - 1)$ is in $Ns$. Since $\alpha -1$ acts invertibly, $s | m$. But by assumption $N \neq Ns$. Thus there exists some $m\in N$ such that $a | m$ but $c$ does not divide $m$.
\end{proof}

\begin{cor}\label{dual}
Let $R$ be a Pr\"{u}fer domain. Let $b, d, \alpha, r, s \in R$, $d\alpha=b r$ and $b(1-\alpha)=d s$. Then
\[\left(xb=0/xd=0\right)=\left(x s =0/x=0\right)\cap\left(d|x / x=0\right)\cap\left(x\alpha=0/x=0\right)=\]
\[= \left(x s =0/x=0\right)\cap\left(x=x / xd=0\right)\cap\left(x\alpha=0/x=0\right)\]

\end{cor}

\begin{proof}
Herzog showed in \cite{Hr} that the standard duality $D$ between the lattices of left and right pp-formulae first defined by Prest (see \cite[8.4]{Preb1}) induces an isomorphism between the lattice of open sets of $\Zg_R$ and that of the left Ziegler spectrum $_R\Zg$ by sending a basic open set $(\phi / \psi)$ to $(D \psi / D \phi)$. Replacing $a, c$ by $d, b$ in the previous lemma (applied to left modules), we get
\[\left(d|x / b|x \right)=\left(x=x/ s | x\right)\cap\left(x=x / dx =0\right)\cap\left(x=x/\alpha|x\right).\]
Since for every $t \in R$, $D(t \midd x)$ is $xt = 0$ and $D(tx = 0)$ is $t \midd x$, we deduce
\[\left(xb=0 / xd= 0 \right)=\left(xs = 0/ x = 0 \right)\cap\left(d \midd x / x=0\right)\cap\left(x\alpha = 0 / x = 0 \right).\]
Observe that $(d \midd x / x = 0) = (x = x / xd = 0)$, since in both cases an indecomposable pure injective $N$ in the corresponding open set is asked to contain an element $m$ with $md \neq 0$.\end{proof}

The next lemma provides a sort of generalization of the final claim of the previous proof.

\begin{lemma}
Let $R$ be a Pr\"ufer domain (actually any commutative ring), and $a, d\in R$. Then
\[\left(a|x/xd=0\right)=\left(x=x/xad=0\right).\]
\end{lemma}
\begin{proof}
Note $ad\in\ann_R N$ if and only if $a|x$ implies $xd=0$ in $N$.
\end{proof}


\begin{lemma}\label{intersection}
Let $R$ be a Pr\"ufer domain, $\theta_j$ pp-formulae ($j \leq m$) and $\mathcal{W},\mathcal{U}_i$ open sets ($1\leq i \leq n$). Then
\begin{enumerate}
\item $\bigcap_{j \leq m} \left( x=x / \theta_j \right) \cap \mathcal{W}\subseteq \bigcup_{i=1}^n \mathcal{U}_i \quad$ if and only if
$\quad \left(x=x / \theta_j \right) \cap \mathcal{W} \subseteq \bigcup_{k \neq j} \left(\theta_k / \theta_j \right) \cup \bigcup_{i=1}^n \mathcal{U}_i$ for all $j \leq m$,
\item $\bigcap_{j \leq m} \left(\theta_j / x=0 \right) \cap \mathcal{W}\subseteq \bigcup_{i=1}^n \mathcal{U}_j \quad$ if and only if
$\quad \left(\theta_j / x=0 \right) \cap \mathcal{W} \subseteq \bigcup_{k \neq j} \left(\theta_j / \theta_k \right) \cup \bigcup_{i=1}^n\mathcal{U}_i$ for all $j \leq m$.
\end{enumerate}
\end{lemma}
\begin{proof}
Use once again pp-uniseriality of indecomposable pure injective modules over $R$.
\end{proof}

Thus in order to show that $T_R$ is decidable and hence to check $(\star)$ it is enough to be able to effectively decide whether
\[\left( x=x / xd=0 \right) \cap \mathcal{W}_1\cap\mathcal{W}_2
\subseteq \bigcup_{i=1}^n \left(\phi_i / \psi_i \right)\]
and
\[\left(xb=0 / c\midd x \right) \cap \mathcal{W}_1\cap\mathcal{W}_2 \subseteq \bigcup_{i=1}^n \left(\phi_i / \psi_i \right)\]
where $\mathcal{W}_1$ and $\mathcal{W}_2$ are of the form $\left( x \alpha = 0 / x=0 \right)$ or $\left( x=x / \delta \midd x \right)$.
In fact the other basic open sets that may arise on the left side, that is, those of the forms $(x = x / c \midd x)$ and $(x b = 0 / x = 0)$ can be absorbed by $\mathcal{W}_1$ and $\mathcal{W}_2$ by Lemma \ref{intersection}. For the same reason, only one set of each kind occurs on the left side of $(\star)$. Furthermore we can assume $b, c, d \neq 0$ in these final statements of $(\star)$.

Similar reductions apply to the right side, where one can assume that, for every $i = 1, \ldots, n$, $(\phi_i / \psi_i)$ is
either $(x=x / x d_i = 0)$ or $(x b_i = 0 / c_i \midd x)$ (replace if necessary in the other cases $x = x$ by $x 0 = 0$ and $x = 0$ by $0 \midd x$).

\section{Localizing}

We still work over a Pr\"ufer domain $R$ (if necessary, effectively given). We continue our analysis of the inclusion $(\star)$ as settled at the end of the last section. To do that, we localize at prime ideals of $R$ and use the results of \cite{Gre15}. We put for simplicity
\[\mathcal{W}_{\lambda, h , g}:=\left(\frac{x\lambda h =0}{g \midd x + x\lambda=0}\right)\]
where $\lambda, g, h$ denote elements of $R$. By this notation we cover all the basic open sets we are interested in, on the left and on the right side of $(\star)$.

The most important case in our analysis is that of $\mathcal{W}_{1, h, g}$. We use the following notation.

\begin{definition}
For $\mathfrak{p},\mathfrak{q}$ prime ideals of $R$, let
\[X_{\mathfrak{p},\mathfrak{q}}:=\{N\in\Zg_R \, \mid \, \Ass N=\mathfrak{p} \text{ and } \Div N=\mathfrak{q}\}.\]
\end{definition}

Recall the following fact.

\begin{remark}
If $\mathfrak{p},\mathfrak{q}$ are prime ideals such that $\mathfrak{p}+\mathfrak{q}\neq R$ then either $\mathfrak{p}\subseteq \mathfrak{q}$ or $\mathfrak{q}\subseteq\mathfrak{p}$.
\end{remark}
\begin{proof}
Since $\mathfrak{p}+\mathfrak{q}\neq R$ there exists a maximal ideal $\mathfrak{m}$ such that $\mathfrak{p},\mathfrak{q}\subseteq\mathfrak{m}$. Thus either $\mathfrak{p}R_{\mathfrak{m}}\subseteq \mathfrak{q}R_{\mathfrak{m}}$ or $\mathfrak{q}R_{\mathfrak{m}}\subseteq\mathfrak{p}R_{\mathfrak{m}}$.
Since $\mathfrak{p},\mathfrak{q}\subseteq \mathfrak{m}$, $\mathfrak{p}R_{\mathfrak{m}}\cap R=\mathfrak{p}$ and $\mathfrak{q}R_{\mathfrak{m}}\cap R=\mathfrak{q}$. So either $\mathfrak{p}\subseteq\mathfrak{q}$ or $\mathfrak{q}\subseteq\mathfrak{p}$.
\end{proof}

The next definition is crucial for our purposes.

\begin{definition}
Let $\mathfrak{p}$ be a prime ideal of $R$. For $a,b\in R\backslash\{0\}$ we set $a\leq_{\mathfrak{p}} b$ if and only if $bR_{\mathfrak{p}} \subseteq aR_{\mathfrak{p}}$. For all $a\in R$ we set $a \leq_\mathfrak{p} 0$.
\end{definition}

\begin{remark}\label{first}
{\rm Let $\mathfrak{p}$ be a prime ideal and $\mathfrak{m}\supseteq \mathfrak{p}$ a maximal ideal. If $a,b\in R$ then $a\leq_{\mathfrak{p}}b$ if only if $a\leq_{\mathfrak{p}R_{\mathfrak{m}}}b$ in $R_{\mathfrak{m}}$ in the sense of \cite[4.18]{Gre15}. This is
because $b\mathfrak{p}R_{\mathfrak{m}}\subseteq a\mathfrak{p}R_{\mathfrak{m}}$ if and only if $bR_{\mathfrak{p}}\subseteq aR_{\mathfrak{p}}$.}
\end{remark}

The relation $\leq_{\mathfrak{p}}$ can be equivalently characterized in the following way, using Fact \ref{tuganbaev}.

\begin{lemma}\label{generalize}
Let $\mathfrak{p}$ be a prime ideal of $R$, $a, b \in R$, $\alpha, r, s \in R$, $b \alpha = as$ and $a (\alpha - 1) = br$. Then $a \leq_{\mathfrak{p}} b$ if and only if $\alpha \notin \mathfrak{p}$ or $r \notin \mathfrak{p}$.
\end{lemma}
\begin{proof}  If $\alpha \notin \mathfrak{p}$ then $\alpha$ is invertible in $R_{\mathfrak{p}}$, so $R_{\mathfrak{p}}$ includes $s / \alpha$ and $b = a (s/ \alpha)\in a R_{\mathfrak{p}}$.
Likewise, if $r \notin \mathfrak{p}$ then $R_{\mathfrak{p}}$ includes $(\alpha - 1)/r$ and hence $b = a (\alpha - 1)/r \in a R_{\mathfrak{p}}$. So we have proved the reverse direction.

Conversely suppose that $\alpha \in \mathfrak{p}$ and $r \in\mathfrak{p}$. Then $\alpha - 1 \notin \mathfrak{p}$ and $\alpha - 1$ is a unit in $R_{\mathfrak{p}}$. Thus $a = b r/(\alpha - 1)$. Since $r \in \mathfrak{p}$, $r / (\alpha - 1) \in \mathfrak{p}R_{\mathfrak{p}}$. It follows $a \in b \mathfrak{p} R_{\mathfrak{p}}$. So $b \notin aR_{\mathfrak{p}}$ since $a R_{\mathfrak{p}} \subseteq b \mathfrak{p}R_{\mathfrak{p}}\subsetneq bR_{\mathfrak{p}}$. Hence we have proved the forward direction.
\end{proof}

Over a B\'ezout domain $R$ a further, simpler characterization can be provided in terms of $\gcd$. For all $a,b\in R$ put
\[\gamma(a,b):=\left\{
                  \begin{array}{ll}
                    a/\gcd(a,b), & \hbox{if  $b\neq 0$,} \\
                    1, & \hbox{if  $b=0$.}
                  \end{array}
                \right.
\]

\begin{lemma}\label{purposeofgamma}
Let $R$ be a B\'ezout domain, $a,b\in R\backslash\{0\}$ and $\mathfrak{p}$ be a prime ideal of $R$. Then
$a \leq_{\mathfrak{p}} b$ if and only if $\gamma(a,b) \notin \mathfrak{p}$.
\end{lemma}
\begin{proof}
If $a / \gcd(a,b)\notin \mathfrak{p}$ then $a / \gcd(a,b)$ is a unit in $R_{\mathfrak{p}}$. Thus $b / \gcd(a,b)R_{\mathfrak{p}}\subseteq a/\gcd(a,b)R_{\mathfrak{p}}$. So $bR_{\mathfrak{p}}\subseteq aR_{\mathfrak{p}}$.

If $a / \gcd(a,b)\in\mathfrak{p}$ then $b / \gcd(a,b)\notin \mathfrak{p}$ since $a/\gcd(a,b)$ and $b/\gcd(a,b)$ are coprime. Thus $b / \gcd(a,b)$ is a unit in $R_{\mathfrak{p}}$. Therefore $a / \gcd(a,b)R_{\mathfrak{p}}\subset \mathfrak{p} \subsetneq b / \gcd(a,b) R_{\mathfrak{p}}$. So $aR_{\mathfrak{p}}\subsetneq bR_{\mathfrak{p}}$.
\end{proof}

Here are the main results of this section, again valid over any Pr\"ufer domain $R$.

\begin{lemma}\label{incpneqq}
Let $\mathfrak{p}\neq\mathfrak{q}$ be prime ideals in $R$ such that $\mathfrak{p}+\mathfrak{q}\neq R$. Let $\mu_i, g_i, h_i,\lambda,g,h\in R$ with
$\lambda, \mu_i \neq 0$, $h_i, h\in\mathfrak{p}$ and $g_i, g\in\mathfrak{q}$ ($1 \leq i \leq n$). Then
\[\mathcal{W}_{\lambda,h,g}\cap X_{\mathfrak{p},\mathfrak{q}}\subseteq\bigcup_{i=1}^n\mathcal{W}_{\mu_i,h_i,g_i}\cap X_{\mathfrak{p},\mathfrak{q}}\] if and only if
\[[\lambda,\lambda gh)_{\mathfrak{p}\cap\mathfrak{q}}\subseteq\bigcup_{i=1}^n[\mu_i,\mu_i g_i h_i)_{\mathfrak{p}\cap\mathfrak{q}}.\]
\end{lemma}
\begin{proof}
Intervals refer to $\leq_{\mathfrak{p} \cap \mathfrak{q}}$.
By working inside $R_{\mathfrak{m}}$ where $\mathfrak{m}$ is a maximal ideal containing $\mathfrak{p}+\mathfrak{q}$, this follows directly from \cite[4.21]{Gre15}.
\end{proof}

\begin{lemma}\label{incpeqq}
Let $\mathfrak{p}$ be a prime ideal in $R$. Let $\mu_i,g_i,h_i,\lambda,g,h\in R$ with $\lambda, \mu_i \neq 0$, $g_i, h_i, g, h\in\mathfrak{p}$ ($1 \leq i \leq n$). Then
\[\mathcal{W}_{\lambda, h, g} \cap X_{\mathfrak{p},\mathfrak{p}}\subseteq \bigcup_{i=1}^n\mathcal{W}_{\mu_i, h_i, g_i}\cap X_{\mathfrak{p},\mathfrak{p}}\] if and only if
$$
(\lambda,\lambda gh)_{\mathfrak{p}}\subseteq\bigcup_{i=1}^n (\mu_i,\mu_i g_i h_i)_{\mathfrak{p}}.$$
\end{lemma}
\begin{proof}
Intervals refer to $\leq_{\mathfrak{p}}$.
By working inside $R_{\mathfrak{m}}$ where $\mathfrak{m}$ is a maximal ideal containing $\mathfrak{p}$, this follows directly from \cite[4.23]{Gre15}.
\end{proof}
As a first consequence we get:
\begin{lemma}\label{pairwisecomplast2var}
Let $g_i, h_i\in R$ for $1 \leq i \leq n$ and let $\mathfrak{p},\mathfrak{q}$ be prime ideals such that  $\mathfrak{p}+\mathfrak{q}\neq R$ (possibly $\mathfrak{p} = \mathfrak{q}$). Then the sets
\[\mathcal{W}_{1, h_i, g_i}\cap X_{\mathfrak{p},\mathfrak{q}}\] ($i = 1, \ldots, n$) are pairwise comparable under inclusion.
\end{lemma}
\begin{proof}
If $\mathcal{W}_{1, h_i, g_i}\cap X_{\mathfrak{p},\mathfrak{q}}$ is empty then it is comparable with every other set. So, take two non-empty set $\mathcal{W}_{1, h_1, g_1}\cap X_{\mathfrak{p},\mathfrak{q}}$ and $\mathcal{W}_{1, h_2, g_2}\cap X_{\mathfrak{p},\mathfrak{q}}$. Then $h_1, h_2\in\mathfrak{p}$ and $g_1, g_2\in \mathfrak{q}$.

Now, since $\leq_{\mathfrak{p}\cap\mathfrak{q}}$ is a total order (on ideals corresponding to elements), either $g_1 h_1\leq_{\mathfrak{p}\cap\mathfrak{q}} g_2 h_2$ or $g_2 h_2\leq_{\mathfrak{p}\cap\mathfrak{q}} g_1 h_1$.

Thus either
$$
(1, h_1 g_1)_{\mathfrak{p}\cap\mathfrak{q}}\subseteq (1,h_2 g_2)_{\mathfrak{p}\cap\mathfrak{q}} \, \, {\rm (respectively} \, \,
[1, h_1 g_1)_{\mathfrak{p} \cap \mathfrak{q}} \subseteq [1, h_2 g_2)_{\mathfrak{p} \cap \mathfrak{q}} {\rm )}
$$
or
$$
(1, h_2 g_2)_{\mathfrak{p} \cap \mathfrak{q}} \subseteq (1,h_1 g_1)_{\mathfrak{p} \cap\mathfrak{q}} \, \, {\rm (respectively}
\, \, [1, h_2 g_2)_{\mathfrak{p} \cap \mathfrak{q}}\subseteq [1,h_1 g_1)_{\mathfrak{p}\cap\mathfrak{q}} {\rm )}.
$$\end{proof}

\begin{prop}\label{mainprop}
Let $\lambda, g, h\in R$ with $\lambda\neq 0$ and $g_i, h_i
\in R$ ($i = 1, \ldots, n$), $\mu_j \in R$, $\mu_j \neq 0$ ($j = 1, \ldots, m)$.
The following are equivalent.
\begin{enumerate}
\item $\mathcal{W}_{\lambda, h, g} \subseteq \bigcup_{i=1}^n\mathcal{W}_{1, h_i, g_i} \cup \bigcup_{j=1}^m \mathcal{W}_{\mu_j,0,0}$.
\item For all prime ideals $\mathfrak{p}$ and $\mathfrak{q}$ of $R$ with
$\mathfrak{p} + \mathfrak{q} \neq R$ (and possibly $\mathfrak{p} = \mathfrak{q}$), $\mathcal{W}_{\lambda, h, g} \cap X_{\mathfrak{p}, \mathfrak{q}} \subseteq \left (\bigcup_{i=1}^n\mathcal{W}_{1, h_i, g_i} \cup \bigcup_{j=1}^m \mathcal{W}_{\mu_j,0,0} \right) \cap X_{\mathfrak{p}, \mathfrak{q}}$.
\item For all prime ideals $\mathfrak{p},\mathfrak{q}$ such that $\mathfrak{p}+\mathfrak{q} \neq R$, $h \in\mathfrak{p}$ and $g \in\mathfrak{q}$, one of the following holds
\begin{enumerate}
\item $\mu_j \leq_{\mathfrak{p} \cap \mathfrak{q}} \lambda$ for some $j =1, \ldots, m$,
\item there exists $i = 1, \ldots, n$ such that $h_i\in\mathfrak{p}$,
$g_i\in\mathfrak{q}$ and $\lambda gh\leq_{\mathfrak{p}\cap\mathfrak{q}} g_i h_i$,
\item there exist $i = 1, \ldots, n$ and $j = 1, \ldots, m$ such that $h_i \in\mathfrak{p}$,
$g_i \in \mathfrak{q}$ and $\mu_j \leq_{\mathfrak{p}\cap\mathfrak{q}} g_i h_i$
\end{enumerate}
and for all primes $\mathfrak{p}$ with $g,h\in\mathfrak{p}$ one of the following holds
\begin{enumerate}
\item $\mu_j \leq_{\mathfrak{p}}\lambda$ for some $j = 1, \ldots, m$,
\item there exists $i = 1, \ldots, n$ such that $g_i, h_i \in \mathfrak{p}$ and $\lambda gh \leq_{\mathfrak{p} } g_i h_i$,
\item there exist $i = 1, \ldots, n$ and $j = 1, \ldots, m$ such that $g_i, h_i\in\mathfrak{p}$ and $\mu_j <_{\mathfrak{p}} g_i h_i$.
\end{enumerate}
\end{enumerate}

\end{prop}


\begin{proof}
$(1)\Rightarrow (3)$ Suppose that $\mathfrak{p},\mathfrak{q}$ are prime ideals, $\mathfrak{p}+\mathfrak{q}\neq R$, $h \in\mathfrak{p}$, $g \in\mathfrak{q}$ and that
\[\mathcal{W}_{\lambda,h,g}\subseteq \bigcup_{i=1}^n\mathcal{W}_{1,h_i,g_i} \cup \bigcup_{j=1}^m \mathcal{W}_{\mu_j,0,0}.\]
First assume $\mathfrak{p} \neq \mathfrak{q}$.
By Lemma \ref{pairwisecomplast2var} this implies (unless $n = 0$, namely no open set $\mathcal{W}_{1, h_i, g_i}$ is involved) that there exists an $1\leq i\leq n$ such that
\[\mathcal{W}_{\lambda,h,g} \cap X_{\mathfrak{p},\mathfrak{q}}\subseteq \left( \mathcal{W}_{1, h_i, g_i}\cup \bigcup_{j=1}^m \mathcal{W}_{\mu_j,0,0} \right) \cap X_{\mathfrak{p},\mathfrak{q}}.\]

Since $\lambda\neq 0$, $h \in \mathfrak{p}$ and $g \in \mathfrak{q}$,
$\mathcal{W}_{\lambda,h,g} \cap X_{\mathfrak{p},\mathfrak{q}}
\neq\emptyset$. Thus by Lemma \ref{incpneqq}, either
$[\lambda,\lambda gh)_{\mathfrak{p}\cap\mathfrak{q}}\subseteq [\mu_j,0)_{\mathfrak{p}\cap\mathfrak{q}}$ for some $j$ or $h_i\in\mathfrak{p}, g_i\in\mathfrak{q}$ and
$[\lambda,\lambda gh)_{\mathfrak{p}\cap\mathfrak{q}}\subseteq [1, g_i h_i)_{\mathfrak{p}\cap\mathfrak{q}}\cup [\mu_j,0)_{\mathfrak{p}\cap\mathfrak{q}}$ for some $j$ (which also holds in the parenthetical case $n=0$). Since $\leq_{\mathfrak{p}\cap\mathfrak{q}}$ is a total order (on principal ideals $r R_{\mathfrak{p} \cap \mathfrak{q}}$ with $r \in R$), either $\mu_j \leq_{\mathfrak{p}\cap\mathfrak{q}}\lambda$, $\lambda gh\leq_{\mathfrak{p}\cap\mathfrak{q}}g_i h_i $ or $\mu_j \leq_{\mathfrak{p}\cap\mathfrak{q}} g_i h_i$.

Now suppose that $\mathfrak{p}$ is prime and $g,h\in\mathfrak{p}$. As before, we can assume that there exists an $1\leq i\leq n$ such that
\[\mathcal{W}_{\lambda, h, g} \cap X_{\mathfrak{p},\mathfrak{p}} \subseteq \left( \mathcal{W}_{1, h_i, g_i}
\cup \bigcup_{j=1}^m \mathcal{W}_{\mu_j,0,0} \right) \cap X_{\mathfrak{p},\mathfrak{p}}.\] By Lemma \ref{incpeqq}, either
$(\lambda,\lambda gh)_{\mathfrak{p}}\subseteq (\mu_j,0)_{\mathfrak{p}}$ for some $j$ or $h_i \in\mathfrak{p}, g_i\in\mathfrak{p}$ and for some $j$
\[(\lambda,\lambda gh)_{\mathfrak{p}}\subseteq (\mu_j,0)_{\mathfrak{p}}\cup (1, g_i h_i)_{\mathfrak{p}}.\] Since $\leq_{\mathfrak{p}}$ is a total order (on principal ideals $r R_{\mathfrak{p}}$ with $r \in R$), either $\mu_j \leq_{\mathfrak{p}}\lambda$, $\lambda gh \leq_{\mathfrak{p}} g_i h_i $ or $\mu_j <_{\mathfrak{p}} g_i h_i$.

\smallskip

$(3)\Rightarrow (2)$
We need to show that for all prime ideals $\mathfrak{p},\mathfrak{q}$ such that $\mathfrak{p}+\mathfrak{q}\neq R$,
\[\mathcal{W}_{\lambda, h, g}\cap X_{\mathfrak{p},\mathfrak{q}}\subseteq \bigcup_{i=1}^n \left( \mathcal{W}_{1, h_i, g_i} \cap X_{\mathfrak{p},\mathfrak{q}} \right) \cup \left( \bigcup_{j=1}^m \mathcal{W}_{\mu_j,0,0} \cap X_{\mathfrak{p},\mathfrak{q}} \right). \]
If $h \notin \mathfrak{p}$ or $g \notin \mathfrak{q}$ then $\mathcal{W}_{\lambda, h, g} \cap X_{\mathfrak{p},\mathfrak{q}}=\emptyset$.
Then we may assume $h \in\mathfrak{p}$ and $g \in\mathfrak{q}$.

First suppose that $\mathfrak{p}\neq \mathfrak{q}$. By (2), one of the following holds
\begin{itemize}
\item[(a)] $\mu_j \leq_{\mathfrak{p} \cap \mathfrak{q}} \lambda$ for some $j$,
\item[(b)] there exists $i = 1, \ldots, n$ such that $h_i \in
\mathfrak{p}$, $g_i \in \mathfrak{q}$ and $\lambda gh \leq_{\mathfrak{p}\cap\mathfrak{q}} g_i h_i$,
\item[(c)] there exist $i = 1, \ldots, n$ and $j = 1, \ldots, m$ such that $h_i \in \mathfrak{p}$, $g_i \in \mathfrak{q}$ and $\mu \leq_{\mathfrak{p}\cap\mathfrak{q}}g_i h_i$.
\end{itemize}

By Lemma \ref{incpneqq} each of (a), (b) and (c) implies
\[\mathcal{W}_{\lambda, h, g} \cap X_{\mathfrak{p},\mathfrak{q}}\subseteq \left( \mathcal{W}_{1, h_i, g_i} \cup \bigcup_{j=1}^m \mathcal{W}_{\mu_j,0,0} \right) \cap X_{\mathfrak{p},\mathfrak{q}}.\]

Now suppose $\mathfrak{p}=\mathfrak{q}$. By (2), one of the following holds
\begin{itemize}
\item[(a)] $\mu_j \leq_{\mathfrak{p}}\lambda$ for some $j$,
\item[(b)] there exists $i = 1, \ldots, n$ such that $g_i, h_i \in \mathfrak{p}$ and $\lambda gh \leq_{\mathfrak{p}} g_i h_i $,
\item[(c)] there exist $i = 1, \ldots, n$ and $j = 1, \ldots, m$ such that $g_i, h_i \in \mathfrak{p}$ and $\mu<_{\mathfrak{p}} g_i h_i$.
\end{itemize}

By Lemma \ref{incpeqq} each of (a), (b) and (c) implies
\[\mathcal{W}_{\lambda, h, g}\cap X_{\mathfrak{p},\mathfrak{p}}\subseteq \left( \mathcal{W}_{1, h_i, g_i} \cup \bigcup_{j=1}^m \mathcal{W}_{\mu_j,0,0} \right) \cap X_{\mathfrak{p},\mathfrak{p}}.\]

Thus for all $\mathfrak{p},\mathfrak{q}$ such that $\mathfrak{p}+\mathfrak{q}\neq R$,
\[\mathcal{W}_{\lambda, h, g} \cap X_{\mathfrak{p},\mathfrak{q}}\subseteq \bigcup_{i=1}^n \left( \mathcal{W}_{1, h_i, g_i} \cap X_{\mathfrak{p},\mathfrak{q}} \right) \cup \left( \bigcup_{j=1}^m \mathcal{W}_{\mu_j,0,0} \cap X_{\mathfrak{p},\mathfrak{q}} \right).\]

So (2) holds.

\smallskip

$(2) \Rightarrow (1)$ This is because every indecomposable pure injective $R$-module $N$ admits some $\mathfrak{p}$ and $\mathfrak{q}$ as $\Ass N$ and $\Div N$ respectively, and $\Ass N + \Div N = \Ass N \cup \Div N \neq R$.\end{proof}

\section{The main theorem}

We state and prove here our main result, concerning B\'ezout domains, that is, Theorem \ref{decforbezbase} below. First we introduce the 4-ary relation characterizing the effectively given B\'ezout domains $R$ for which $T_R$ is decidable. We call it the {\sl double
prime radical relation}, written $\DPR$. It makes sense for every commutative ring $R$, in particular for Pr\"ufer domains. 

We define $\DPR (R)$ to be
\[\{(a,b,c,d)\in R^4 \mid \text{ for all prime ideals } \mathfrak{p}, \mathfrak{q}\subseteq R \text{ with } \mathfrak{p}+\mathfrak{q}\neq R \text{ either }a\in\mathfrak{p}, b\notin \mathfrak{p}, c\in\mathfrak{q},\text{ or }d\notin\mathfrak{q}\}.\]

Here are three characterizations of $\DPR$ over Pr\"ufer domains. The first applies to a wider framework. It uses localization.

\begin{prop}\label{maximal}
Let $R$ be a commutative domain, $a, b, c, d \in R$. Then the following are equivalent
\begin{enumerate}
\item $(a, b, c, d) \notin \DPR(R)$,
\item there is some maximal ideal $\mathfrak{m}$ of $R$ such that $a \notin \rad (b R_{\mathfrak{m}})$ and $c \notin \rad (d R_{\mathfrak{m}})$.
\end{enumerate}
\end{prop}
\begin{proof}
$(1) \Rightarrow (2)$ Let $\mathfrak{p},\mathfrak{q}$ be proper prime ideals such that $\mathfrak{p}+\mathfrak{q}\neq R$, $a\notin \mathfrak{p}$, $b\in \mathfrak{p}$, $c\notin \mathfrak{q}$ and $d\in \mathfrak{q}$. Let $\mathfrak{m}$ be a maximal ideal of $R$ extending $\mathfrak{p} + \mathfrak{q}$. Thus $\mathfrak{p} R_{\mathfrak{m}}$ and $\mathfrak{q}R_{\mathfrak{m}}$ are ideals of $R_{\mathfrak{m}}$. The former includes $b$ but not $a$, and the latter includes $d$ but not $c$. Thus $a \notin \rad (b R_{\mathfrak{m}})$ and $c \notin \rad (d R_{\mathfrak{m}})$.

$(2) \Rightarrow (1)$ There are two prime ideals of $R_{\mathfrak{m}}$, the former including $b$ but not $a$, and the latter including $d$ but not $c$. These ideals can be represented as $\mathfrak{p} R_{\mathfrak{m}}$ and $\mathfrak{q}R_{\mathfrak{m}}$ for some (unique) prime ideals $\mathfrak{p}$ and ${\mathfrak{q}}$ in ${\mathfrak{m}}$. Clearly $a \notin \mathfrak{p}$, $b \in \mathfrak{p}$, $c \notin \mathfrak{q}$ and $d \in \mathfrak{q}$.
\end{proof}

The second characterization directly refers to a Pr\"ufer domain
$R$. Its statement does not involve localization. We need the following premise, that should be well known.

\begin{lemma}\label{premise}
Suppose that $R$ is a Pr\"ufer domain, $\mathfrak{p}$ is a prime ideal and $r \in \mathfrak{p}$. Then  $\rad(r R)_\mathfrak{p}$ is a prime ideal of the localization  $R_\mathfrak{p}$.
\end{lemma}
\begin{proof}
Suppose that $ab\in \rad(rR)_\mathfrak{p}$, i.e. $(ab)^n s\in rR$ for some $s\notin \mathfrak{p}$ and some positive integer $n$. We may assume that $a\in b R_\mathfrak{p}$, hence $a^{2n}\in r R_\mathfrak{p}$.
\end{proof}

\begin{prop}\label{rad2}
Let $R$ be a Pr\"ufer domain. Then the following are equivalent for $a, b, c, d \in R$:
\begin{enumerate}
\item $(a,b,c,d) \notin \DPR(R)$.
\item $(\rad (bR):a) + (\rad (dR) :c)$ is a proper ideal of $R$.
\end{enumerate}
\end{prop}
\begin{proof}
(2) $\Ra$ (1) Let $I= (\rad (bR) : a)$, $J= (\rad (dR): c)$ and choose a maximal ideal $\mathfrak{m}$ containing $I+ J$. The localization $R_{\mathfrak{m}}$ is a commutative valuation domain.

Set $\mathfrak{p} = \rad(b)_{\mathfrak{m}} \cap R$ and $\mathfrak{q} = \rad(d)_{\mathfrak{m}}\cap R$, both are prime ideals of $R$. By the definition we have $b\in \rad(bR)\seq \mathfrak{p} $ and $d\in \rad (dR)\seq \mathfrak{q} $.

Assume $1 \in \mathfrak{p} + \mathfrak{q}$, so $1 = u + v$ with
$u, v \in R$, $s u \in \rad(bR)$ and $t v \in \rad(dR)$ for some $s, t \in R \backslash \mathfrak{m}$. Then $su \in (\rad(bR) : a) \seq \mathfrak{m}$, whence $u \in \mathfrak{m}$. Similarly $v \in \mathfrak{m}$. But then $1 \in \mathfrak{m}$. Thus we have proved that $\mathfrak{p} + \mathfrak{q} \neq R$.

Suppose, by a way of contradiction, that $a\in \mathfrak{p} $. This means that $as\in \rad(bR)$ for some $s \notin \mathfrak{m}$, i.e.
$s\in (\rad (bR):a)$. By the definition $s\in I\seq \mathfrak{m}$, a contradiction. Thus $a\notin \mathfrak{p}$, and similarly we
conclude that $c \notin \mathfrak{q}$.

(1) $\Ra$ (2) Suppose that (1) holds but $(\rad (bR):a)+ (\rad (dR): b)= R$.

Since $\mathfrak{p} + \mathfrak{q}  \subset R$ choose a maximal ideal $\mathfrak{m}$ containing $\mathfrak{p} , \mathfrak{q}$ and localize: let $V$ be the commutative valuation domain $R_{\mathfrak{m}} $. Without loss of generality we may assume that $(\rad (bR) :a)_{\mathfrak{m}}= V$, i.e.
$as\in \rad(bR)$ for some $s \notin \mathfrak{m}$. It follows that $a^n\cdot s^n\in bR\seq \mathfrak{p}$ for some $n$. Since
$s^n\notin \mathfrak{m}$, we conclude that $a\in \mathfrak{p}$, a contradiction.
\end{proof}

Here is the third characterization.

\begin{prop}\label{Xisrec}
Let $R$ be a Pr\"{u}fer domain. The following are equivalent for $a, b, c, d \in R$:
\begin{enumerate}
\item $(a,b,c,d)\in \DPR(R)$,
\item $(xb=0 \, / \, d \midd x) \subseteq (xa=0 \, /  \, x=0) \cup
(x=x \, / \, c \midd x)$.
\end{enumerate}
\end{prop}
\begin{proof}
Suppose $(a,b,c,d) \in \DPR(R)$. Let $N \in (xb=0 \, / \, d \midd x)$. Then $b\in \Ass N$, $d\in\Div N$ and $\Ass N+\Div N\neq R$. Thus either $a\in \Ass N$ or $c\in\Div N$. So either
$N\in \left(xa=0 \, / \, x=0\right)$ or $N\in \left(x=x \, / \, c \midd x \right)$.

Conversely suppose that (2) holds and that $\mathfrak{p},\mathfrak{q}$ are prime ideals such that $\mathfrak{p}+\mathfrak{q}\neq R$, $b\in\mathfrak{p}$, $d\in\mathfrak{q}$ and $a\notin\mathfrak{p}$. We need to show that $c\in \mathfrak{q}$.

Let $\mathfrak{m}$ be a maximal ideal of $R$ containing $\mathfrak{p}+\mathfrak{q}$. Then the indecomposable pure injective $R_\mathfrak{m}$-module $N$ corresponding to the pair $(\mathfrak{p}R_{\mathfrak{m}},\mathfrak{q}R_{\mathfrak{m}})$ is in $(xb=0 \, / \, d
\midd x)$ over $R_\mathfrak{m}$. Since $a \notin \mathfrak{p}$, $a\notin \mathfrak{p}R_{\mathfrak{m}}$. Thus $N\notin \left(xa=0 \, / \, x=0\right)$. Therefore $N\in \left(x=x \, / \, c \midd x\right)$. So $c\in \mathfrak{q}R_{\mathfrak{m}}$. Thus $c\in\mathfrak{q}$ as required.
\end{proof}

\begin{theorem}\label{decforbezbase}
Let $R$ be an effectively given B\'ezout domain with all its residue fields infinite. Then the common theory $T_R$ of all $R$-modules is decidable if and only if there is an algorithm which, given $a,b,c,d\in R$ answers whether $(a,b,c,d) \in \DPR (R)$ or not.
\end{theorem}

The proof uses the following notion: a \textbf{condition on a pair of prime ideals} is a condition of the form $a\in P$ or $b\in Q$ where $a,b\in R$ and $(P,Q)$ is a variable for pair of prime ideals. We will say that a pair of prime ideals $(\mathfrak{p},\mathfrak{q})$ satisfies the condition $a\in P$ if $a\in\mathfrak{p}$ and satisfies the condition $b\in Q$ if $b\in \mathfrak{q}$.


\begin{lemma}\label{pairprimes}
Let $R$ be a(n effectively given) B\'ezout domain, and $\Delta$ be a Boolean combination of conditions on a pair of prime ideals.  If $\DPR(R) \subseteq R^4$ is recursive, then there is an algorithm which answers whether for all prime ideals $\mathfrak{p},\mathfrak{q}$, $\mathfrak{p}+\mathfrak{q}\neq R$ implies that $(\mathfrak{p},\mathfrak{q})$ satisfies $\Delta$.
\end{lemma}
\begin{proof}
By putting $\Delta$ into conjunctive normal form, we may assume that $\Delta$ is of the form $\bigwedge_{h=1}^m\Delta_h$ where
\[\Delta_h:=\bigvee_{i\in I_h}a_{hi}\in P\vee\bigvee_{j\in J_h} b_{hj}\notin P \vee \bigvee_{k\in K_h} c_{hk}\in Q \vee \bigvee_{l\in L_h} d_{hl}\notin Q\] where $I_h,J_h,K_h,L_h$ are finite sets and $a_{hi}, b_{hj}, c_{hk}, d_{hl}\in R$.

A pair of prime ideals $(\mathfrak{p},\mathfrak{q})$ satisfies $\Delta_h$ if and only if
\[\prod_{i\in I_h}a_{hi}\in \mathfrak{p} \text{ or } \gcd(b_{hj})_{j\in J_h}\notin \mathfrak{p} \text{ or } \prod_{i\in K_h}c_{hk}\in \mathfrak{q} \text{ or }\gcd(d_{hl})_{l\in L_h}\notin \mathfrak{q}.\]

Therefore, for all pairs of prime ideals $(\mathfrak{p},\mathfrak{q})$, $\mathfrak{p}+\mathfrak{q}\neq R$ implies $(\mathfrak{p},\mathfrak{q})$ satisfies $\Delta$ if and only if
\[(\prod_{i\in I_h}a_{hi}, \gcd(b_{hj})_{j\in J_h}, \prod_{k\in K_h}c_{hk}, \gcd(d_{hl})_{l\in L_h})\in \DPR(R)\] for all $1\leq h\leq m$.
\end{proof}


\begin{proof}[Proof of \ref{decforbezbase}]
The forward direction follows from Proposition \ref{Xisrec}. In fact, using it, one can check membership to $\DPR$ provided that one
can decide inclusions like $(\star)$.

Now suppose that $\DPR(R) \subseteq R^4$ is recursive. We look for an algorithm deciding inclusions of basic open sets of $\Zg_R$ \[\left( x=x / xd=0 \right) \cap \mathcal{W}_1\cap\mathcal{W}_2
\subseteq \bigcup_{i=1}^n \left(\phi_i / \psi_i \right),\]
\[\left(xb=0 / c\midd x \right) \cap \mathcal{W}_1\cap\mathcal{W}_2 \subseteq \bigcup_{i=1}^n \left(\phi_i / \psi_i \right)\]
as at the end of $\S$ 4 -- hence $\mathcal{W}_1$ and $\mathcal{W}_2$ are of the form $\left( x \alpha = 0 / x=0 \right)$ or $\left( x=x / \delta \midd x \right)$, $c, d \neq 0$, $b \neq 0$ (otherwise $xb = 0$ gets equivalent to $x=x$) and, for every $i = 1, \ldots, n$, $(\phi_i / \psi_i)$ can be assumed to equal
either $(x=x / x d_i = 0)$ or $(x b_i = 0 / c_i \midd x)$.

Considering these inclusions intersected with $X_{\mathfrak{p},\mathfrak{q}}$
where $\mathfrak{p}$ and $\mathfrak{q}$ are (possibly equal) prime ideals of $R$ and $\mathfrak{p}+\mathfrak{q}\neq R$, it is enough to effectively decide, given $\alpha_1, \alpha_2, \beta_1, \beta_2 \in R$ and $d, b_i, c_i, d_i$ also in $R$, whether, for all prime ideals $\mathfrak{p},\mathfrak{q}$ satisfying the previous assumptions and $\alpha_1,\alpha_2\in\mathfrak{p}$, $\beta_1,\beta_2\in \mathfrak{q}$,
\[\left(\frac{x=x}{xd=0}\right) \cap X_{\mathfrak{p},\mathfrak{q}}\subseteq \bigcup_{i=1}^n \left(\frac{\phi_i}{\psi_i}\right) \cap X_{\mathfrak{p},\mathfrak{q}} \]
and, given $\alpha_1,\alpha_2,\beta_1,\beta_2\in R$ and $b, c, b_i, c_i, d_i \in R$, whether for all prime ideals $\mathfrak{p},\mathfrak{q}$ satisfying the same condition as before
\[\left(\frac{xb=0}{c\midd x}\right)\cap X_{\mathfrak{p},\mathfrak{q}}\subseteq \bigcup_{i=1}^n \left(\frac{\phi_i}{\psi_i} \right) \cap X_{\mathfrak{p},\mathfrak{q}}.\]
These cases can be effectively handled because $\DPR(R)$ is recursive.

In view of Lemma \ref{purposeofgamma}, Proposition \ref{mainprop} implies that, in the more general setting corresponding to $\lambda, g, h\in R$, $g_i, h_i\in R$ ($1\leq i\leq n$) and $\mu_j \in R$ ($1 \leq j \leq m$ with $\lambda, \mu_j
\neq 0$ we can decide whether
\[\mathcal{W}_{\lambda, h, g}\subseteq \bigcup_{i=1}^n\mathcal{W}_{1, h_i, g_i}\cup \bigcup_{j=1}^m \mathcal{W}_{\mu_j,0,0} , \] or also whether
\[\mathcal{W}_{\lambda, h, g} \cap X_{\mathfrak{p}, \mathfrak{q}} \subseteq \left (\bigcup_{i=1}^n\mathcal{W}_{1, h_i, g_i} \cup \bigcup_{j=1}^m \mathcal{W}_{\mu_j,0,0} \right) \cap X_{\mathfrak{p}, \mathfrak{q}} \]
for all prime ideals $\mathfrak{p}$ and $\mathfrak{q}$ of $R$ with
$\mathfrak{p} + \mathfrak{q} \neq R$, if we can effectively decide whether for all $\mathfrak{p},\mathfrak{q}$ as before,
$\mathfrak{p} + \mathfrak{q}\neq R$ implies that a particular condition on a pair of prime ideals holds. By Lemma \ref{pairprimes}
and the hypothesis that $\DPR(R)$ is recursive in $R^4$ we can effectively decide such conditions.
\end{proof}

As a consequence we get the following strengthening of Theorem \cite[Thm. 3.4]{LPT} -- a key step towards the decidability result for
the theory of modules over the ring of algebraic integers.
\begin{cor}\label{krull1}
Let $R$ be an effectively given B\'ezout domain of Krull dimension $1$ all of whose residue fields are infinite. The theory of $R$-modules is decidable.
\end{cor}

\begin{proof}
By \cite[Lemma 3.3]{LPT} (using the Krull dimension 1 hypothesis) the prime radical relation $a \in \rad (bR)$ can be decided effectively for $a, b \in R$.

On the other hand Theorem \ref{decforbezbase}, when applied to a B\'ezout domain of Krull dimension 1, just says that the theory of $R$-modules is decidable if and only if there is an algorithm which given $a, b \in R$ decides whether $a \in \rad (bR)$.

Let us explain why. Since $R$ has Krull dimension $1$, all non-zero prime ideals are maximal. Thus if $\mathfrak{p},\mathfrak{q}$ are non-zero prime ideals then $\mathfrak{p}+\mathfrak{q} \neq R$ if and only if $\mathfrak{p}=\mathfrak{q}$. It follows that, for $a, b, c, d \in R$, $(a,b,c,d) \in \DPR(R)$  if and only if the following conditions hold:
\begin{enumerate}
\item for all prime ideals $\mathfrak{p}$, $a\in\mathfrak{p}$, $b\notin\mathfrak{p}$, $c\in\mathfrak{p}$ or $d\notin\mathfrak{p}$
\item for all prime ideals $\mathfrak{p}$, $a\in\mathfrak{p}$, $b\notin\mathfrak{p}$, $c=0$ or $d\neq 0$
\item for all prime ideals $\mathfrak{p}$, $a=0$, $b\neq 0$, $c\in\mathfrak{p}$ or $d\notin\mathfrak{p}$.
\end{enumerate}

The first condition is equivalent to $ac\in\mathfrak{p}$ or $\gcd(b,d)\notin\mathfrak{p}$ for all $\mathfrak{p}$, which is equivalent in its turn to $ac\in\rad(\gcd(b,d)R)$.

The second condition is equivalent to $a\in\rad(bR)$ or $c=0$ or $d\neq 0$.

The third condition is equivalent to $a=0$, $b\neq 0$, or $c\in\rad(dR)$.
\end{proof}


\begin{exam}\label{ex1}
Corollary \ref{krull1} also applies to the B\'ezout domain $R$ associated by the Kaplansky-Jaffard-Ohm construction to the subgroup $\Gamma$ of $\mathbb{Z}^{\mathbb{N}}$ consisting of the eventually constant sequences (see \cite[Ex. 5.5 p. 114]{F-S}) and to an infinite effectively given field $K$. Recall that $\Gamma$ and indeed $\mathbb{Z}^{\mathbb{N}}$ are lattice ordered abelian groups under the pointwise order relation. The elements of $\Gamma$ can be represented as finite ordered sequences (of any length $t+1$) of integers $(a_0, \ldots, a_{t-1}, a_t)$, meaning that the corresponding infinite sequences stabilize to $a_t$ after $t$ terms. On this basis it is easily seen that $\Gamma$ is effectively given as a lattice ordered abelian group. In particular the lattice operation $\land$ reduces to taking the minimum pointwise.
By Proposition \ref{group} $R$ itself is effectively given. Moreover it has Krull dimension 1 (see again \cite[p. 113]{F-S}) and infinite residue fields. As said, this also implies that the radical relation is recursive.

In our specific case, for $a, b \in R$, $a \in \rad(bR)$ holds if and only if there is some positive integer $n$ such that $n c(a) = c(a^n) \geq c(b)$, where $c$ denotes content. Checking this condition on the finite sequences of integers $(a_0, \ldots, a_t)$, $(b_0, \ldots, b_{t'})$ corresponding to $c(a)$, $c(b)$ and with respect to the pointwise order relation is a straightforward procedure.
\end{exam}

\begin{exam}\label{ex2}
A similar case, but with Krull dimension $>1$, is that of \cite[Ex. 6.7 p. 119]{F-S}, see also \cite{BCM74}. This time $\Gamma$ is introduced as a lattice ordered subgroup of the direct product $\mathbb{Z}^{\mathbb{N}}$ ordered by putting, for every $r = (r_n)_{n \in \mathbb{N}}$, $r \geq 0$ if and only if either $r_0 > 0$ and $r_n \geq 0$ for all $n \geq 2$, or $r_0 = 0$ and $r_n \geq 0$ for all $n \geq 1$ (thus the first two components are ordered lexicographically, hence totally, while their pairs and the remaining components are ordered pointwise). To be precise, let $\Gamma$ be the subgroup generated by the direct sum $\mathbb{Z}^{(\mathbb{N} \backslash \{ 0 \})}$ and the element $s = (1, 0, 1, 0, \ldots)$ with the induced ordering. Then the elements of $\Gamma$ have the form $(r_0, r_1, r_2 + r_0, r_3, r_4 + r_0, \ldots, 0, r_0, 0, r_0, \ldots )$. It follows that $\Gamma$ is again effectively given as a lattice ordered group, whence the associated B\'ezout domain (with respect to an effectively given field $K$) is also effectively given. Actually $R$ is presented in \cite{BCM74} as an example of a B\'ezout domain which is not adequate but each non-zero prime ideal is contained in a unique maximal ideal. Indeed every non-zero prime ideal is also maximal, with exactly one exception, given by a chain of length 2 of non-zero prime ideals $\mathfrak{p}_0 \subsetneq \mathfrak{p}_1$ with $\mathfrak{p}_1$ maximal. Notice that these ideals
$\mathfrak{p_0}$ and $\mathfrak{p}_1$ are explicitly described in terms of
the associated prime filters in the positive part $\Gamma^+$ of $\Gamma$
(see \cite{BCM74}) via the correspondence between these two settings (as explained, for instance, in \cite[Prop. 4.6 p. 110]{F-S} and \cite[pp. 196-199]{Gi}). In fact, both these prime filters and the correspondence
itself between ideals of $R$ and filters of $\Gamma^+$ can be in their turn effectively described with respect to the explicit representation of elements of $\Gamma$ and $R$.

Now observe that, for $\mathfrak{p}$ and $\mathfrak{q}$ prime ideals of $R$,
$\mathfrak{p} + \mathfrak{q} \neq R$ if and only if either $\mathfrak{p} = \mathfrak{q}$, or $\mathfrak{p} =0$, or $\mathfrak{q} = 0$, or $\mathfrak{p} = \mathfrak{p}_0$ and $\mathfrak{q} = \mathfrak{p}_1$, or vice versa $\mathfrak{p} = \mathfrak{p}_1$ and $\mathfrak{q} = \mathfrak{p}_0$. It follows that, for
$a, b, c, d \in R$, $(a, b, c, d) \notin \DPR(R)$ (that is, there are $\mathfrak{p}$ and $\mathfrak{q}$ such that $\mathfrak{p} + \mathfrak{q} \neq R$, $a \notin \mathfrak{p}$, $b \in \mathfrak{p}$, $c \notin \mathfrak{q}$ and $d \in \mathfrak{q}$) if and only if either
\begin{enumerate}
\item there is a prime ideal $\mathfrak{p} \neq 0$ containing $b, d$ and excluding $a, c$, or
\item there is a prime ideal $\mathfrak{p} \neq 0$ containing $b$ and not $a$, and $c \neq 0$, $d= 0$ (or a similar condition swapping $a, b$ and $c, d$), or
\item $a \notin \mathfrak{p}_0$, $b \in \mathfrak{p}_0$, $c \notin \mathfrak{p}_1$, $d \in \mathfrak{p}_1$ (or again a similar condition swapping $a, b$ and $c, d$).
\end{enumerate}
The first two cases can be handled as in Example \ref{ex1}. The third can be checked using the effective representations of $\mathfrak{p}_0$ and
$\mathfrak{p}_1$.

Therefore $\DPR$ can be effectively checked provided that the radical relation $a \in \rad(b R)$ is recursively answered for $a, b \in R$. But this can be done more or less as in the previous case. Moreover, choosing $K$ infinite ensures that the residue fields of $R$ are infinite.

Note that the same argument applies to every effectively given B\'ezout domain $R$ such that residue fields are infinite, almost all prime
ideals are maximal, the finitely many other prime ideals are
contained in only finitely many maximal ideals and finally
all height 2 maximal ideals are recursive, as well as all
the height 1 prime non-maximal ideals.
\end{exam}

\section{From B\'ezout to Pr\"ufer}

We extend here our analysis to Pr\"ufer domains $R$ and we partly generalize to their setting the main theorem of the last section. We need a larger family of \lq \lq prime radical" relations in addition to $\DPR$. For every positive integer $n$ we introduce a $(2n+2)$-ary relation
\begin{multline*}
 \DPR_n(R):= \{(a, c, b_1, \ldots, b_n, d_1, \ldots, d_n)\in R^{2n+2} \mid \text{ for all prime ideals } \mathfrak{p}, \mathfrak{q}\subseteq R \text{ with } \\  \mathfrak{p}+\mathfrak{q}\neq R  \text{ either }a\in\mathfrak{p}, c\in \mathfrak{q}, b_i\notin\mathfrak{p}\text{ for some }1\leq i\leq n,\text{ or }d_i\notin\mathfrak{q}\text{ for some }1\leq i\leq n\}.
\end{multline*}
Hence $\DPR$ is just $\DPR_1$.

\begin{theorem}\label{pruefer}
Let $R$ be an effectively given Pr\"ufer domain with an infinite residue field for every maximal ideal. If there are algorithms deciding in $R$ the membership to $\DPR_n$ for every positive integer $n$ (uniformly for all $n$) then the theory $T_R$ of all $R$-modules is decidable.
\end{theorem}

\begin{proof}
Most of the argument working over B\'ezout domains also applies to Pr\"ufer domains. However we need to be careful about the steps involving $\gcd$ - indeed just one, that is, Lemma \ref{pairprimes} about Boolean combinations of conditions of pairs of prime ideals. That result remains valid, provided that we strengthen its assumptions and require that all the relations $\DPR_n$ are recursive in $R$. In fact, without $\gcd$, the various conditions $b_{hj} \notin P$, or $d_{hl} \notin Q$ cannot be joined in single statements $\gcd (b_{hj})_{j\in J_h} \notin P$, $\gcd (d_{hl})_{l\in K_h} \notin Q$. On the other hand, one can assume without loss of generality that there are as many $j$'s as $l$'s (otherwise add some $1$'s as $b_{hj}$ or $d_{hl}$). Therefore, if all the $\DPR_n(R)$ are recursive (uniformly for all $n$), then we can decide the truth value of all Boolean combinations of conditions on a pair of prime ideals.
\end{proof}

This theorem raises two questions which we were not able to answer: first, is the condition on the $\DPR_n$ not only sufficient but also necessary to guarantee that $T_R$ is decidable? And secondly, can we bound the $n$'s to check, as in the B\'ezout case where $n=1$ is enough?

%

\begin{cor}\label{krullN}
Let $R$ be a Pr\"{u}fer domain all of whose residue fields are infinite and $N$ be a positive integer such that any finitely generated ideal of $R$ can be generated by $N$ elements (in particular, this is the case when $R$ has Krull dimension $N-1$). If there are algorithms deciding membership of $\DPR_N(R)$ then $T_R$ is decidable.
\end{cor}
\begin{proof}

Clearly, if we can effectively decide $\DPR_N(R)$, then the same is true of $\DPR_n(R)$ for every $n \leq N$. This is because $(a,c,b_1,\ldots,b_n,d_1,\ldots,d_n)\in \DPR_n(R)$ if and only if $$(a,c,b_1,\ldots,b_n,\underbrace{b_n,\ldots,b_n}_{N-n \text{ times }},d_1,\ldots,d_n,\underbrace{d_n,\ldots,d_n}_{N-n \text{ times }})\in \DPR_N(R).$$

Now suppose that $n > N$ and $a,c,b_1,\ldots b_n,d_1,\ldots,d_n\in R$. Since $R$ is effectively given and all finitely generated ideals can be generated by $N$ elements, we can effectively find $b_1',\ldots,b_N'\in R$ and $d_1',\ldots,d_N'\in R$ such that $\sum_{j=1}^n b_j R = \sum_{i=1}^N b_i'R$ and $\sum_{n=1}^n d_j R = \sum_{i=1}^N d_i'R$. On the other hand $(a,c,b_1,\ldots b_n,d_1,\ldots,d_n)\in \DPR_n(R)$ if and only if $(a,c,b_1',\ldots b_N',d_1',\ldots,d_N')\in \DPR_N(R)$.

Finally note that Heitmann shows in \cite{Hei76} that if a Pr\"{u}fer domain has Krull dimension $d$ then every finitely generated ideal can be generated by $d+1$ elements.
\end{proof}

Next we provide a partial converse to this result. We need the following easy fact.

\begin{lemma}\label{decidablerelations}
Let $R$ be a commutative ring. If $T_R$ is decidable then
there is an algorithm which given $a,b_1,...,b_n\in R$ answers whether $a\in\rad(b_1R+...+b_nR)$.
\end{lemma}


\begin{proof}
We claim that $a\in \rad(b_1R+...+b_nR)$ if and only if
\[\exists x (x\neq 0\wedge \bigwedge_{i=1}^nxb_i=0)\rightarrow \exists x(x\neq 0\wedge xa=0)\in T_R.\]

Suppose that $a\in \rad(b_1R+...+b_nR)$. Then there exists a positive integer $k$ and $r_i\in R$ for $1\leq i\leq n$ such that $a^k = \sum_{i=1}^nb_ir_i$. Let $M$ be a module over $R$ with an element $m \neq 0$ satisfying $mb_i=0$ for $1\leq i\leq n$. Then $ma^k=0$. Thus there exists $l\in \N$, $l < k$ such that $(ma^{l})a=0$ and $ma^l\neq 0$.

Conversely, suppose that
\[\exists x (x\neq 0\wedge \bigwedge_{i=1}^nxb_i=0)\rightarrow \exists x(x\neq 0\wedge xa=0)\in T_R.\] If $\mathfrak{p}$ is a prime ideal and $b_1,\ldots,b_n\in \mathfrak{p}$ then $0=(1+\mathfrak{p}) \, b_i\in R/\mathfrak{p}$. Thus there exists $r\in R\backslash\mathfrak{p}$ such that $0=(r+\mathfrak{p})a\in R/\mathfrak{p}$. Hence $ra\in\mathfrak{p}$ and so $a\in \mathfrak{p}$. Thus $a\in\rad(b_1R+...+b_nR)$.
\end{proof}

\begin{prop}
Let $R$ be a Pr\"{u}fer domain of Krull dimension $1$ all of whose residue fields are infinite. Then the following are equivalent:
\begin{enumerate}
\item $T_R$ is decidable.
\item $\DPR_2(R)$ is recursive.
\item There is an algorithm which given $a,b_1,b_2\in R$ answers whether $a\in\rad(b_1R+b_2R)$.
\end{enumerate}
\end{prop}
\begin{proof}
$(1) \Ra (3)$ This is a particular case of Lemma \ref{decidablerelations}, (2).

$(2) \Ra (1)$ This is a special case of Corollary \ref{krullN}.


$(3) \Ra (2)$ Since $R$ has Krull dimension $1$ all non-zero prime ideals are maximal. Thus if $\mathfrak{p},\mathfrak{q}$ are non-zero prime ideals then $\mathfrak{p}+\mathfrak{q} \neq R$ if and only if $\mathfrak{p}=\mathfrak{q}$. It follows that, for $a, c, b_1, b_2, d_1, d_2 \in R$, $(a,c,b_1,b_2,d_1,d_2) \in \DPR_2(R)$  if and only if the following conditions hold:
\begin{enumerate}
\item for all prime ideals $\mathfrak{p}$, $a\in\mathfrak{p}$, $b_1\notin\mathfrak{p}$, $b_2\notin \mathfrak{p}$, $c\in\mathfrak{p}$, $d_1\notin\mathfrak{p}$ or $d_2\notin\mathfrak{p}$
\item for all prime ideals $\mathfrak{p}$, $a\in\mathfrak{p}$, $b_1\notin\mathfrak{p}$, $b_2\notin\mathfrak{p}$, $c=0$, $d_1\neq 0$ or $d_2\neq 0$
\item for all prime ideals $\mathfrak{p}$, $a=0$, $b_1\neq 0$, $b_2\neq 0$, $c\in\mathfrak{p}$, $d_1\notin\mathfrak{p}$ or $d_2\notin\mathfrak{p}$.
\end{enumerate}

The first condition is equivalent to $ac\in\mathfrak{p}$ or $b_1R+b_2R+d_1R+d_2R\nsubseteq\mathfrak{p}$ for all $\mathfrak{p}$. This is equivalent to $ac\in\rad(b_1R+b_2R+d_1R+d_2R)$.

The second condition is equivalent to $a\in\rad(b_1R+b_2R)$, $c=0$, $d_1\neq 0$ or $d_2\neq 0$.

The third condition is equivalent to $a=0$, $b_1\neq 0$, $b_2\neq 0$, or $c\in\rad(d_1R+d_2R)$.

Finally note that since $R$ is effectively given and all finitely generated ideals can be generated by two elements, we can effectively find $e_1,e_2$ such that $e_1R+e_2R=b_1R+b_2R+d_1R+d_2R$. Hence we are done.\end{proof}

\end{document}